\documentclass[a4paper,10pt]{amsart}

\usepackage{amssymb}
\usepackage{amscd}
\usepackage{enumerate}
\usepackage{hyperref}
\usepackage[utf8]{inputenc}
\usepackage{newunicodechar}
\usepackage{varioref}
\usepackage[arrow,curve,matrix]{xy}
\usepackage{bussproofs}

\usepackage{colortbl}
\usepackage{graphicx}
\usepackage{tikz}

\usepackage{mathrsfs}

\definecolor{linkred}{rgb}{0.7,0.2,0.2}
\definecolor{linkblue}{rgb}{0,0.2,0.6}

\numberwithin{figure}{section}

\usepackage[hyperpageref]{backref}

\newunicodechar{α}{\ensuremath{\alpha}}
\newunicodechar{β}{\ensuremath{\beta}}
\newunicodechar{χ}{\ensuremath{\chi}}
\newunicodechar{δ}{\ensuremath{\delta}}
\newunicodechar{ε}{\ensuremath{\varepsilon}}
\newunicodechar{Δ}{\ensuremath{\Delta}}
\newunicodechar{η}{\ensuremath{\eta}}
\newunicodechar{γ}{\ensuremath{\gamma}}
\newunicodechar{Γ}{\ensuremath{\Gamma}}
\newunicodechar{ι}{\ensuremath{\iota}}
\newunicodechar{κ}{\ensuremath{\kappa}}
\newunicodechar{λ}{\ensuremath{\lambda}}
\newunicodechar{Λ}{\ensuremath{\Lambda}}
\newunicodechar{μ}{\ensuremath{\mu}}
\newunicodechar{ν}{\ensuremath{\nu}}
\newunicodechar{ω}{\ensuremath{\omega}}
\newunicodechar{Ω}{\ensuremath{\Omega}}
\newunicodechar{π}{\ensuremath{\pi}}
\newunicodechar{φ}{\ensuremath{\phi}}
\newunicodechar{Φ}{\ensuremath{\Phi}}
\newunicodechar{ψ}{\ensuremath{\psi}}
\newunicodechar{Ψ}{\ensuremath{\Psi}}
\newunicodechar{ρ}{\ensuremath{\rho}}
\newunicodechar{σ}{\ensuremath{\sigma}}
\newunicodechar{Σ}{\ensuremath{\Sigma}}
\newunicodechar{τ}{\ensuremath{\tau}}
\newunicodechar{θ}{\ensuremath{\theta}}
\newunicodechar{Θ}{\ensuremath{\Theta}}
\newunicodechar{ξ}{\ensuremath{\xi}}
\newunicodechar{ζ}{\ensuremath{\zeta}}

\newcommand{\sC}{\mathcal{C}}

\newcommand{\sE}{\mathcal{E}}

\newcommand{\sI}{\mathcal{I}}

\newcommand{\sK}{\mathcal{K}}

\newcommand{\sM}{\mathcal{M}}

\newcommand{\OO}{\mathcal{O}}

\newcommand{\sS}{\mathcal{S}}

\newcommand{\sU}{\mathcal{U}}
\newcommand{\sV}{\mathcal{V}}
\newcommand{\sW}{\mathcal{W}}
\newcommand{\sX}{\mathcal{X}}
\newcommand{\sY}{\mathcal{Y}}

\newcommand{\ol}[1]{\overline{#1}}

\newcommand{\hU}{\ol{U}}
\newcommand{\hV}{\ol{V}}
\newcommand{\hW}{\ol{W}}
\newcommand{\hX}{\ol{X}}
\newcommand{\hY}{\ol{Y}}
\newcommand{\hZ}{\ol{Z}}

\newcommand{\iX}{X^\circ}
\newcommand{\iY}{Y^\circ}

\newcommand{\mA}{A^\infty}

\newcommand{\mU}{U^\infty}
\newcommand{\mV}{V^\infty}
\newcommand{\mW}{W^\infty}
\newcommand{\mX}{X^\infty}
\newcommand{\mY}{Y^\infty}

\newcommand{\Zar}{\textrm{Zar}}
\newcommand{\et}{\textrm{{\'e}t}}

\newcommand{\Nis}{\operatorname{Nis}}
\newcommand{\sNis}{\operatorname{sNis}}

\newcommand{\Bl}{\operatorname{Bl}}

\newcommand{\smod}{\operatorname{mod}}
\newcommand{\lmod}{\operatorname{lmod}}

\newcommand{\length}{\operatorname{length}}
\renewcommand{\div}{\operatorname{div}}

\newcommand{\Sm}{\mathrm{Sm}}

\newcommand{\Pro}{\mathrm{Pro}}
\newcommand{\PSm}{\mathrm{\underline{P}Sm}}
\newcommand{\mSm}{\mathrm{mSm}}
\newcommand{\mmH}{\mathrm{mH}}
\newcommand{\mmSH}{\mathrm{mSH}}
\newcommand{\MSm}{\mathrm{\underline{M}Sm}}

\newcommand{\PSCH}{\mathrm{\underline{P}SCH}}

\newcommand{\qc}{\mathrm{qc}}

\newcommand{\SmlSm}{\mathrm{SmlSm}}

\newcommand{\MDM}{\operatorname{\underline{M}DM}}
\newcommand{\MH}{\operatorname{\underline{M}H}}
\newcommand{\MSH}{\operatorname{\underline{M}SH}}
\newcommand{\Sp}{\operatorname{Sp}}
\newcommand{\logH}{\operatorname{logH}}

\newcommand{\logDM}{\operatorname{logDM}}

\newcommand{\PSh}{\operatorname{PSh}}
\newcommand{\Shv}{\operatorname{Shv}}

\newcommand{\Cor}{\operatorname{Cor}}
\newcommand{\MCor}{\operatorname{\underline{M}Cor}}
\newcommand{\lCor}{\operatorname{lCor}}
\newcommand{\CI}{\operatorname{CI}}
\newcommand{\BI}{\operatorname{BI}}

\newcommand{\Map}{\operatorname{Map}}

\newcommand{\Spec}{\mathrm{Spec}}

\newcommand{\red}{\mathrm{red}}

\theoremstyle{theorem}
\newtheorem{theo}{Theorem}[section]
\newtheorem{coro}[theo]{Corollary}
\newtheorem{lemm}[theo]{Lemma}
\newtheorem{prop}[theo]{Proposition}

\theoremstyle{definition}
\newtheorem{defi}[theo]{Definition}

\newtheorem{rema}[theo]{Remark}

\newtheorem{exam}[theo]{Example}

\DeclareSymbolFontAlphabet{\scr}{rsfs}

\newcommand{\NN}{\mathbb{N}}
\newcommand{\PP}{\mathbb{P}}
\newcommand{\QQ}{\mathbb{Q}}

\newcommand{\ZZ}{\mathbb{Z}}
\renewcommand{\AA}{\mathbb{A}}

\DeclareMathOperator{\id}{id}

\DeclareMathOperator{\colim}{colim}

\newcommand{\bcube}{\overline{\square}}
\newcommand{\PrL}{\mathcal{P}{\operatorfont{r^L}}}

\newcommand{\nc}{\mathrm{}}

\title{
Log homotopy types are \\ homotopy types with modulus
}
\author{Shane Kelly}
\date{\today}

\begin{document}

\maketitle

\begin{abstract}
We show that the category of log homotopy types is a full subcategory of a category of homotopy types with modulus.
\end{abstract}

\section{Introduction}

Morel and Voevodsky's $\AA^1$-homotopy theory has many successes, most famously the proof of the Norm Residue Theorem 
comparing Milnor $K$-theory to étale cohomology with coefficients away from the characteristic. However, by definition it cannot see non-$\AA^1$-invariant phenomena such Hodge cohomology or Hochschild homology. It also fails see to invariants that are non-$\AA^1$-invariant adjacent such as wild ramification.

In recent years a number of larger categories have been proposed such as motives with modulus, \cite{KMSYIII}, log motives, \cite{BPO22}, and motivic spectra, \cite{AI22}, \cite{AHI24}. 

\emph{Motives with modulus} are designed to study cohomology groups of a variety $X$, such as $H_{\et}^1(X, \QQ / \ZZ)$ or $H^n_{\Zar}(X, \Omega^m)$, whose classes can be bounded by something like ramification order or pole order; the bound is expressed by a divisor $D$. The rôle of $\AA^1$ is replaced by the pair $(\PP^1, \infty)$, that is, the space of cohomology classes on $\PP^1$ with ``tame ramification'' or ``simple poles'' at infinity. \emph{Log motives} are designed to study cohomology groups of log schemes. The rôle of $\AA^1$ is replaced by the log scheme $(\PP^1, \OO_{\PP^1} \cap \OO_{\AA^1}^*)$, that is, the compactifying log structure for the open immersion $\AA^1 \subseteq \PP^1$. \emph{Motivic spectra} is an approach that systematically replaces the use of $\AA^1$ in the Morel--Voevodsky theory with $\PP^1$-stability and smooth blowup excision. Motivic spectra do not appear in this note.

In this note we compare an unstable homotopy theory of modulus pairs $\MH_k^{\nc}$ with the unstable homotopy theory of log schemes $\logH_k$, filling in the details of one of the comparisons suggested to Miyazaki and Saito in \cite{Kel20}. Heuristically, the equivalence in Theorem~\ref{theo:one} expresses the idea that log cohomology is the first level of the ``ramification'' or ``pole'' filtration, namely, the ``tame ramification'' or ``simple pole'' level. The notation is explained as the introduction progresses culminating in Eq.\eqref{equa:MHkintro} and Eq.\eqref{equa:logHintro} on Page~\pageref{equa:MHkintro}.

\begin{theo}[{Theorem~\ref{theo:main}}] \label{theo:one}
The canonical comparison functor induces an equivalence%
\footnote{%
Here $(-)[\Lambda^{-1}]$ means localisation in the category $\PrL$ of presentable $\infty$-categories and colimit preserving functors, \cite[%
Def.5.2.7.2, 
Def.5.5.3.1, 
Prop.5.5.4.20
]{HTT}. 
}%
\[ \MH_k^{\nc}[\Lambda^{-1}] \stackrel{\sim}{\to} \logH_k
\]
where $\Lambda$ is the class of (images of) morphisms of the form 
$$(\hX, n \mX) \to (\hX, \mX)$$ for $(\hX, \mX) \in \PSm_k^{\nc}$, Def.\ref{defi:PSmnc} and $n \geq 1$. In particular, there is a fully faithful embedding
\[ \MH_k^{\nc} \supseteq \logH_k \]
admitting a left adjoint. 
\end{theo}

Theorem~\ref{theo:one} is a shadow of a stronger equivalence. We write $\PSm^{\nc}_k$ and $\SmlSm_k$ for the categories of pairs with normal crossings support and smooth log smooth schemes respectively, Def.\ref{defi:PSmnc}, Def.\ref{defi:SmlSmk}. So $\PSm^{\nc}_k$ generates $\MH_k^{\nc}$ and $\SmlSm_k$ generates $\logH_k$. The objects of $\PSm^{\nc}_k$ are pairs $\sX = (\hX, \mX)$ consisting of a smooth $k$-scheme $\hX$ and an effective Cartier divisor $\mX \subseteq \hX$ with strict normal crossing support, Def.\ref{defi:sncSupport}. 
Morphisms $\sX \to \sY$ are morphisms of schemes $f: \hX \to \hY$ such that $\mX \supseteq \mY|_{\hX}$. We write $\iX = \hX \setminus \mX$ for the \emph{interior} and 
\[ \log\sX = (\hX, \OO_X \cap \OO_{\iX}^*) \]
for the associated log scheme.

\begin{prop}[{Proposition~\ref{prop:PSmLambdaSmlSm}}] \label{prop:introSchEquiv}
The functor $\sX \mapsto \log \sX$ induces an equivalence
\[ \PSm^{\nc}_k[\Lambda^{-1}] \stackrel{\sim}{\to} \SmlSm_k \]
where $\Lambda$ is the class of morphisms of the form $(\hX, n \mX) \to (\hX, \mX)$ and this localisation happens in the category of small $\infty$-categories, cf.~Def.\ref{defi:catloc}. 
\end{prop}

\begin{rema}
A qcqs base version of Theorem~\ref{theo:one} is likely to be true because a qcqs base version of Propostion~\ref{prop:introSchEquiv} holds, cf.~Remark~\ref{rema:qcbase}.
\end{rema}

\begin{rema}
The motives version, $\MDM_k[\Lambda^{-1}] \stackrel{?}{=} \logDM_k$ of Theorem~\ref{theo:one} may be true. However, the cycles version of Propostion~\ref{prop:introSchEquiv} is false, cf.Example~\ref{exam:gmgm}.
\end{rema}

\begin{rema}
In case it's not obvious, we point out that the same techniques show that Morel--Voevodsky's unstable homotopy category can also be as the localisation
\[ \MH_k[I^{-1}] \stackrel{\sim}{\to} \logH_k[(\log I)^{-1}] \stackrel{\sim}{\to} H_k \]
where $I$ is the class of morphisms of the form $(\hX \setminus \mX, \varnothing) \to (\hX, \mX)$.
\end{rema}

A consequence of Theorem~\ref{theo:one} is a scheme theoretic description of log homotopy types. That is, a description which does not explicitly use any log geometry. This is not so surprising since the log schemes $\log \sX$ in question are completely determined by the total space $\hX$ and the interior $\iX$, cf.\cite[Prop.III.1.6.2]{Ogu18}.

\begin{coro}
There is an identification
\[ \logH_k \cong \PSh(\PSm_k^{\nc}, \sS) \left [\bigl (\Nis \cup \Sigma \cup \CI \cup \Lambda \bigr)^{-1} \right] \]
where
\begin{enumerate}
 \item $\PSh(-, \sS)$ means presheaves of spaces.
 \item $\Nis$ is the class of morphisms defining Nisnevich excision, Rem.\ref{rema:NisExcis}.
 \item $\Sigma$ is the class of compositions of smooth blowups, Def.\ref{defi:sbu}.
 \item $\CI$ is the class of morphisms of the form $(\hX \times \PP^1, \mX + \{\infty\})  \to (\hX, \mX)$.
 \item $\Lambda$ is the class of morphisms of the form $(\hX, n \mX) \to (\hX, \mX)$ for $(\hX, \mX) \in \PSm_k^{\nc}$.
\end{enumerate}
\end{coro}

As an afterthought, we show how the stable category of modulus pairs with $\QQ$-divisors $\mmH_k$ from \cite{KMS25} can be built directly from our unstable, integral divisor category $\MH_k^{\nc}$.

Recall that the abelian group of rational numbers can be expressed as the union $\QQ = \cup_{n \in \NN} \tfrac{1}{n}\ZZ$, or equivalently, as the colimit $\QQ = \colim_{n \in \NN} \ZZ$ where the transition morphism $\ZZ \to \ZZ$ associated to the indexing arrow $n {\to} nm$ is multiplication by $m$.

\begin{theo}[{Theorem~\ref{prop:mH}}] \label{theo:1.6}
Let $\mmH_k$ be the category from \cite{KMS25} recalled in Section~\ref{sec:ratDiv}, and $\CI(n) \subseteq \MH_k$ the class of morphisms of the form $(\hX \times \PP^1, \mX + n\{\infty\})  \to (\hX, \mX)$, Def.~\ref{defi:CIn}. Then there is a canonical equivalance
\[ 
\mmH_k
\cong
\underset{n \in \NN}{\colim} \MH_k^{\nc}[\CI(n)^{-1}]
\]
Here, the transition morphism in the colimit associated to $n {\to} nm$ is induced by the functor $(-)^{(m)}: \PSm_k^{\nc} \to \PSm_k^{\nc}; (\hX, \mX) \mapsto (\hX, m \mX)$. Moreover, if $k$ has resolution of singularities, then the canonical functor is fully faithful;
\[ \MH_k^{\nc} \subseteq \mmH_k. \]
\end{theo}

It seems likely that many results about $\mmH_k$ (and therefore the stable version $\mmSH_k = \Sp(\mmH_k)$) actually come from results about $\MH_k^{\nc}$ via Theorem~\ref{theo:1.6}.

The proof of 
Theorem~\ref{theo:one} is straightforward and can be understood from the following outline. The strategy, essentially, is to insert the definitions.

\subsection{Outline} In Section~\ref{sec:schemes} we recall the categories $\PSm^{\nc}_k$ and $\SmlSm_k$ of pairs with normal crossings support and smooth log smooth schemes respectively, Def.\ref{defi:PSmnc}, Def.\ref{defi:SmlSmk}. We observe that the canonical comparison functor induces an equivalence
\[ \PSm^{\nc}_k[\Lambda^{-1}] \stackrel{\sim}{\to} \SmlSm_k \]
from the localisation along the class $\Lambda$ of morphisms of the form $(\hX, n \mX) \to (\hX, \mX)$, Prop.\ref{prop:PSmLambdaSmlSm}. %

In Section~\ref{sec:Nis} we recall the Nisnevich topology, Def.\ref{defi:PNis}, and strict Nisnevich topology, Def.\ref{defi:logNis}, on the categories $\PSm^{\nc}_k$ and $\SmlSm_k$ respectively, the fact that descent can be detected by excision, Rem.\ref{rema:excisionDescent}, and the corollary that the corresponding localised categories are equivalent, Prop.\ref{prop:NisSame},
\[ 
\Shv_{\Nis}(\PSm^{\nc}_k)[\Lambda^{-1}]
\stackrel{\sim}{\to} 
\Shv_{\sNis}(\SmlSm_k) 
\]

In Section~\ref{sec:adb} we recall the notion of smooth blowup in $\PSm_k$, Def.\ref{defi:sbu}, and in $\SmlSm_k$, Def.\ref{defi:logsbu}, and observe that the comparison functor between the localised categories is an equivalence, Prop.\ref{prop:MSmSmlSmSigma}.

\[ 
\Shv_{\Nis}(\PSm_k)[\Sigma^{-1}][\Lambda^{-1}] 
\stackrel{\sim}{\to} 
\Shv_{\sNis}(\SmlSm_k)[\Sigma^{-1}].
\]
In Section~\ref{sec:homCat} we state the definitions of the homotopy categories as 
\begin{align}
\MH_k^{\nc} &= \Shv_{\Nis}(\PSm_k^{\nc})[\Sigma^{-1}, \CI^{-1}]  \label{equa:MHkintro} \\
\logH_k &= \Shv_{\sNis}(\SmlSm_k)[\Sigma^{-1}, (\log\CI)^{-1}] \label{equa:logHintro}
\end{align}
and observe that there is an equivalence
\[ \MH_k^{\nc}[\Lambda^{-1}] \stackrel{\sim}{\to} \logH_k. \]


In Section~\ref{sec:ratDiv} we prove the comparison with $\QQ$-divisors, Thm.\ref{theo:1.6} ($=$ Thm.\ref{prop:mH}).

In Section~\ref{sec:transfers} we discuss transfers and show that the cycles version of Propostion~\ref{prop:introSchEquiv} is false, Exam.\ref{exam:gmgm}.

\emph{Related work.} Koizumi-Miyazaki-Saito prove a stable version of Theorem~\ref{theo:one} with $\QQ$-divisors in \cite{KMS25}. For more about this see Section~\ref{sec:ratDiv}.

\emph{Acknowledgements.} I thank Federico Binda and Doosung Park for answering many questions about log motives. I thank Bruno Kahn and Shuji Saito for asking questions  about an earlier version of this note. I apologise to Miyazaki who, due to poor communication, wrote up a version of Theorem~\ref{theo:one} independently.

\section{Log schemes are modulus pairs with completed modulus} \label{sec:schemes}

Given a morphism of schemes $T \to S$ and a $S$-scheme $X$, we will use the notational convenience
\[ X|_T := T \times_S X. \]

The log theory is developed using strict normal crossing divisors, cf. \cite[Def.7.2.1, Lem.A.5.10]{BPO22}. Let us recall what this means.

\begin{defi} \label{defi:sncSupport}
\label{defi:snc}
We will say that a closed subscheme $\mX \subseteq \hX$ has strict normal crossing support if there exists a Zariski covering $\{U_i \to X\}_{i \in I}$ and étale morphisms $θ_i: U_i \to \AA^d$ such that 
$\mX|_{U_i} = \mA_{e_1, \dots, e_d}|_{U_i}$ where $\mA_{e_1, \dots, e_d} = \Spec(k[t_1, \dots, t_d] / \langle t_1^{e_1}\dots t_d^{e_d} \rangle)$ with all $e_i \geq 0$.
\end{defi}


\begin{defi} \label{defi:PSmnc}
Let $k$ be a field. The category 
\[ \PSm^{\nc}_k \]
has as objects pairs $\sX = (\hX, \mX)$ such that $\hX$ is a smooth $k$-scheme and $\mX \subseteq \hX$ is an effective Cartier divisor with
normal crossing support.
Morphisms $\sX \to \sY$ are morphisms of schemes $f: \hX \to \hY$ such that $\mX \supseteq \mY|_{\hX}$.
\end{defi}

\begin{exam}
One writes $\bcube$ for the pair $(\PP^1, \{\infty\}) \in \PSm^{\nc}_k$. More generally, for any pair $\sX = (\hX, \mX)$ one writes $\bcube_{\sX}$ for the pair $(\hX \times \PP^1, \mX {\times} \PP^1 + \hX {\times} \{\infty\})$.
\end{exam}

\begin{rema}
As mentioned in the introduction, the divisor $\mX$ should be thought of as a bound on the ramification locus (algebraic setting) or singularity locus (holomorphic setting) of an as yet unspecified collection of objects over $\hX$. 

The direction of the inequality aligns with the canonical inclusion of line bundles $\OO(\mX) \supseteq \OO(\mY|_{\hX})$ where $\OO(\mX) = \mathcal{H}om_{\OO_{\hX}}(\sI_{\mX}, \OO_{\hX})$ is the sheaf of meromorphic functions on $\hX$ with poles bounded by $\mX$. 
That is, this choice of inequality makes 
\[ \sX \mapsto \Gamma(\hX, \OO(\mX)) \]
a presheaf on $\PSm_k$.

\end{rema}

\begin{defi}[{\cite[Def.III.1.1.1]{Ogu18}, cf.\cite[Def.A.2.1, Def.A.2.5]{BPO22}}]
A \emph{log scheme} $(X, α)$ is a scheme $X$ equipped with a Zariski sheaf of additive commutative monoids $M_X$ and a morphism $α: M_X \to \OO_X$ towards the sheaf of multiplicative monoids $\OO_X$ such that $α^{-1}\OO_X^* \to \OO_X^*$ is an isomorphism.
A morphism of log schemes $(X, α) \to (X', α')$ a morphism of schemes $f: X \to X'$ together with a morphism of sheaves of monoids $M_{X'} \to f_*M_X$ compatible with the $α$'s.
\end{defi}

\begin{rema}
Sections of $\sM_X$ are thought of as formal logarithms of functions on $X$ and $α$ as an exponential function, cf. $\exp({m + n}) = \exp(m)\exp(n)$. 
\end{rema}

%
%

\begin{exam} \label{exam:logXD}
Given a scheme $X$ and closed subscheme $D \subseteq X$ with dense open complement $j:U\to X$, we write 
\[ \log(X, D) := (X, \OO_X \cap j_*\OO_U^*) \]
for the associated log scheme. Lot structures occurring in this way are called \emph{compactifying}, \cite[\S III.1.6]{Ogu18}.
\end{exam}

Compactifying log schemes are particularly accessible. 

\begin{lemm}[Cf. {\cite[Prop.III.1.6.2]{Ogu18}}] \label{lemm:logloglemma}
Suppose $p: X \to X'$ is morphism of schemes equipped with nowhere dense closed subschemes $D \subseteq X$, $D' \subseteq X'$ and open complements $U, U'$.
\begin{enumerate}
 \item The canonical morphism induces a bijection
\begin{equation} \notag 
\hom\!\!\!\!\!\!_{\substack{\textrm{log}\\\textrm{schemes}}}\!\! \biggl (\log(X, D), \log(X', D') \biggr ) \stackrel{\sim}{\to} \biggl \{ f \in \hom_{schemes}(X, X') \ |\ f(U) \subseteq U' \biggr \}.
\end{equation}

 \item \label{lemm:logloglemma3} If $D, D'$ are locally principal and $X$ is qc, then the two sets in the previous display equation 
are identified with
 \[ \biggl \{ f \in \hom_{schemes}(X, X') \ |\ \sI_D^n \subseteq \sI_{D'}\OO_X \textrm{ for some }n \in \NN \biggr \}.
\]
\end{enumerate}
\end{lemm}

\begin{proof}
The first claim is a special case of \cite[Prop.III.1.6.2]{Ogu18}. One can also prove it by hand without much trouble. For the second claim, it suffices to show that $U \subseteq U'|_{X'}$ if and only if $\sI_D^n \subseteq \sI_{D'}\OO_X$ for some $n$. Since $X$ is qc and $\sI_D$ locally finitely generated, the condition $\lceil \sI_D^n \subseteq \sI_{D'}\OO_X; \exists n \rfloor$ is equivalent to the condition $\lceil \sI_D \subseteq \sqrt{\sI_{D'}\OO_X} \rfloor$, which is equivalent to the condition $\lceil D \supseteq (D'|_{X})_{\red} \rfloor$ which is equivalent to the condition $\lceil U \subseteq U'|_X \rfloor$.
\end{proof}

\begin{defi}[{cf.\cite[Lem.A.5.10]{BPO22}}] \label{defi:SmlSmk}
The category 
\[ \SmlSm_k \]
has as objects log schemes of the form $\log(X, D)$, Exam.\ref{exam:logXD}, such that $(X, D) \in \PSm^{\nc}_k$. It is a full subcategory of the category of log $k$-schemes.
\end{defi}

For $n \geq 1$ consider the functor
\begin{align} \label{equa:nPower}
(-)^{(n)}: \PSm^{\nc}_k &\to \PSm^{\nc}_k;  \notag \\
(\hX, \mX) &\mapsto (\hX, n\mX)
\end{align}
equipped with its canonical natural transformation $\sX^{(n)} \to \sX$, where we consider $\mX$ as a divisor. That is, in terms of sheaves of ideals, $\mathcal{I}_{n\mX} = \mathcal{I}_{\mX}^n$.

\begin{prop} \label{prop:PSmLambdaSmlSm}
The canonical morphism
\[ \log: \PSm^{\nc}_k \to \SmlSm_k \]
identifies $\SmlSm_k$ with the categorical localisation
\begin{equation} \label{equa:PSmLsmLsm}
\PSm^{\nc}_k[\Lambda^{-1}] \stackrel{\sim}{\to} \SmlSm_k
\end{equation}
where $\Lambda = \{\sX^{(n)} \to \sX\ |\ \sX \in \PSm^{\nc}_k, n \geq 1\}$.
\end{prop}

To be precise:

\begin{defi} \label{defi:catloc}
A functor $C \to C'$ between small $\infty$-categories 
is a \emph{categorical localisation at a class of morphisms $\Lambda$ of $C$} 
if given any other small $\infty$-category $D$, the induced morphism of groupoids of functors
\[ \mathrm{Fun}(C[\Lambda^{-1}], D) \to \mathrm{Fun}(C, D) \]
is fully faithful with essential image consisting of those functors which send elements of $\Lambda$ to isomorphisms.
\end{defi}

\begin{rema}
Since the category small 1-categories sits fully faithfully inside the $\infty$-category of small $\infty$-categories, if $C$ and $C'$ are 1-categories, then the $\infty$-categorical version of Definition~\ref{defi:catloc} implies the 1-categorical version. 

Conversely, if $C$ is a small 1-category and $\Lambda$ satisfies a right calculus of fractions, then the 1-categorical $C[\Lambda^{-1}]$ is also the $\infty$-categorical $C[\Lambda^{-1}]$, \cite[\S 7]{DK80}. 

In our case, one can also prove this latter fact by hand by identifying $\PSm_k[\Lambda^{-1}]$ with the essential image under Yoneda in $\Shv_{λ}(\PSm_k)$ where $λ$ is the topology generated by $\Lambda$.
\end{rema}

\begin{proof}
As mentioned in the remark, if $\Lambda$ satisfies a right calculus of fractions, then a model for the $\infty$-categorical localisation is given by the (1-)category which has the same objects $\PSm_k$ and hom sets 
\[ 
\hom_{\PSm_k[\Lambda^{-1}]}(\sX, \sY)
=
\colim_{n \geq 1} \hom_{\PSm^{\nc}_k}(\sX^{(n)}, \sY).
\]
So let's observe that $\Lambda$ satisfies a right calculus of fractions. Certainly, identities are in $\Lambda$ since $\sX^{(1)} = \sX$. The right Ore condition follows from the natural transformation $(-)^{(n)} \to \id$, and right cancellability follows from the fact that that the ``forget the modulus'' functor $\PSm_k \to \Sm_k$; $\sX \mapsto \hX$ is faithful and $\sX^{(n)} \to \sX$ is sent to $\hX = \hX$.

With the model for $\PSm_k[\Lambda^{-1}]$ just described, the functor \eqref{equa:PSmLsmLsm} is essentially surjective by Definition~\ref{defi:SmlSmk}. To show fully faithfulness, it suffices to show that
\begin{equation} \label{equa:colimTo}
\colim_{n \geq 1} \hom_{\PSm^{\nc}_k}(\sX^{(n)}, \sY) \to \hom_{\SmlSm_k}(\log \sX, \log \sY)
\end{equation}
is an isomorphism where the colimit is over the poset $\NN_{\geq 1}$ ordered by divisibility. Equation~\eqref{equa:colimTo} being an isomorphism follows directly from Lemma~\ref{lemm:logloglemma}. Indeed, morphisms $\sX^{(n)} \to \sY$ in $\PSm^{\nc}_k$ are defined as morphisms of schemes $\hX \to \hY$ such that $n\mX \supseteq \mY|_{\hX}$ and by Lemma~\ref{lemm:logloglemma}\eqref{lemm:logloglemma3} morphisms $\log \sX \to \log \sY$ are identified with morphisms of schemes $\hX \to \hY$ such that $n \mX \supseteq \mY|_{\hX}$ for some $n$.
\end{proof}

\begin{rema} \label{rema:qcbase}
The above argument actually proves the stronger result that, in the notation of \cite[Def.1.4]{KM21}, the canonical functor
\[ \log: \PSCH^{\qc} \to \{\textrm{ log schemes } \} \]
identifies $\PSCH^{\qc}[\Lambda^{-1}]$ with a full subcategory of the category of log schemes
\[ \PSCH^{\qc}[\Lambda^{-1}] \stackrel{\textrm{full}}{\subseteq} \{\textrm{ log schemes } \}. \]
More generally, for any full subcategory $C$ of $\PSCH^{\qc}$ closed under the $(-)^{(n)}$ we get a fully faithful embedding $C[\Lambda^{-1}] \subseteq \{$ log schemes $\}$.
\end{rema}

\begin{rema}
The title of this section refers to the fact that $\PSm^{\nc}_k[\Lambda^{-1}]$ can be identified with a full subcategory of the category $\Pro(\PSm^{\nc}_k)$ of pro-objects, and the functor $\PSm^{\nc}_k \to \PSm^{\nc}_k[\Lambda^{-1}]$ can be identified with the functor $\PSm^{\nc}_k \to \Pro(\PSm^{\nc}_k)$ which sends $\sX$ to $``\lim_{n \in \NN}" \sX^{(n)}$.
\end{rema}

\section{Nisnevich Topologies} \label{sec:Nis}

\begin{defi}
A commutative square of schemes
\[ \xymatrix{
W \ar[r] \ar[d] & V \ar[d] \\
U \ar[r] & X
} \]
is a \emph{distinguished Nisnevich square} if $U \to X$ is an open immersion, $V \to X$ is an étale morphism and $Z|_V \to Z$ is an isomorphism for some (equivalently every) closed subscheme structure on $Z = X \setminus U$.  
\end{defi}

\begin{defi} \label{defi:PNis}
A commutative square (left)
\[ \xymatrix{
\sW \ar[r] \ar[d] & \sV \ar[d] 
&&
\hW \ar[r] \ar[d] & \hV \ar[d] 
\\
\sU \ar[r] & \sX
&&
\hU \ar[r] & \hX
} \]
in $\PSm^{\nc}_k$ is a \emph{distinguished Nisnevich square} if the underlying square of schemes (right) is a distinguished Nisnevich square, and all morphisms are \emph{minimal} in the sense that $\mW = \mX|_{\hW}$, $\mV = \mX|_{\hV}$, and $\mU = \mX|_{\hU}$.
\end{defi}

\begin{defi} \label{defi:PNis}
The \emph{Nisnevich topology} on $\PSm^{\nc}_k$ is generated by\footnote{We include the empty covering of the empty pair.} families of the form $\{\sU \to \sX, \sV \to \sX\}$ associated to distinguished Nisnevich squares. We write $\Nis$ for the Nisnevich topology.
\end{defi}


For the log version we use a log analogue of reduced morphisms of pairs.

\begin{defi}[{\cite[Def.III.1.1.5]{Ogu18}}] \label{defi:logInd} 
Given a morphism of schemes $f: Y \to X$, a log structure $M_X \to \OO_X$ on $X$ induces a log structure $f^*_{\log}M_X \to \OO_Y$ on $Y$. Explicitly, one takes 
\[ f^*_{\log}M_X := f^{-1}M_X \sqcup_{γ^{-1}\OO_Y^*} \OO_Y^* \]
where $γ: f^{-1}M_X \to f^{-1}\OO_X \to \OO_Y$ is the canonical morphism, and the pushout happens in the category of sheaves of monoids, \cite[Prop.II.1.1.5]{Ogu18}. A morphism of log schemes of the form 
$ (Y, f^*_{\log}M_X) \to (X, M_X) $
is called \emph{strict}.
\end{defi}

\begin{rema}
The pushout in Definition~\ref{defi:logInd} is the log structure associated to the prelog structure $f^{-1}M_X \to \OO_Y$.
\end{rema}

\begin{rema} \label{rema:indSub}
If $M_X \to \OO_X$ is injective, then so is $f^{-1}M_X \to f^{-1}\OO_X$. If $f^{-1}\OO_X \to \OO_Y$ is also injective, e.g., $Y \to X$ is a dominant morphism of integral schemes,\footnote{In this case we are dealing with a submorphism of the inclusion $f^{-1}\sK_X \to \sK_Y$.} then $γ: f^{-1}M_X \to \OO_Y$ is injective. It follows that $f^*_{\log}M_X$ is is the subsheaf of monoids of $\OO_Y$ generated by $f^{-1}M_X$ and $\OO_Y^*$,%
\footnote{For this one can locally apply the criterion \cite[Prop.I.1.1.5(2)]{Ogu18} to (our) case where $Q_1, Q_2 \subseteq R$ are submonoids, $Q_2, R$ are groups, and $P = Q_1 \cap Q_2$.
}
\[ f^*_{\log}M_X = \OO_Y^* \cdot f^{-1}M_X \subseteq \OO_Y. \]
\end{rema}

The following is likely to be obvious to readers familiar with log geometry but we include it for the convenience of the novice log geometer.\footnote{Such as the author.}

\begin{lemm} \label{lemm:strictReduced}
Suppose that $\sY \to \sX$ is a reduced morphism, i.e., $\mY = \mX|_{\hY}$, in $\PSm^{\nc}_k$ such that $f: \hY \to \hX$ is étale. 
Then the corresponding morphism $\log \sY \to \log \sX$ in $\SmlSm_k$ is strict, i.e., $\OO_{\hY} \cap \OO_{\iY}^* = f_{\log}^*(\OO_{\hX} \cap \OO_{\iX}^*)$.
\end{lemm}

\begin{proof}
First note that if $k[t_1, \dots, t_d] \to A$ is an étale morphism and $A$ is local, then each $t_i$ is sent to a prime element of $A$. Indeed, the induced morphism $k[t_1, \dots, \hat{t}_i, \dots, t_{d}] \to A / \langle t_i \rangle$ is an étale homomorphism from a regular ring so the target is also regular. Since $A$ is local, so is $A/ \langle t_i \rangle$, so we deduce that $A/ \langle t_i \rangle$ is regular and connected, in particular, integral, so $\langle t_i \rangle$ is prime.

Returning to the lemma, the question is Zariski local so suppose that $A \to B$ is a local étale homomorphism between UFD local rings, $k[t_1, \dots, t_d] \to A$ is an étale morphism, and equip everything with the divisor coming from $t_1^{e_1} \dots t_c^{e_c}$ for some $e_i > 0$. That is, we are considering the morphism of pairs $(\Spec(B), \langle b_1^{e_1} \dots b_c^{e_c} \rangle) \to (\Spec(A), \langle a_1^{e_1} \dots a_c^{e_c} \rangle)$ where $a_i, b_i$ are the images of $t_i$. Since the $a_i, b_i$ are prime, the corresponding log structures are the log structures associated to the prelog rings $\oplus_{i = 1}^c a_i^{\NN} \to A$ and $\oplus_{i = 1}^c b_i^{\NN} \to B$. But since $a_i \mapsto b_i$, the prelog structure $\oplus_{i = 1}^c b_i^{\NN} \to B$ is precisely the composition  $\oplus_{i = 1}^c a_i^{\NN} \to A \to B$ (this corresponds to the morphism $γ$ in Example~\ref{defi:logInd}). Hence, the log structure on $\log(\Spec(B), \langle b_1^{e_1} \dots b_c^{e_c} \rangle)$ is 
the one induced by the log structure on $\log(\Spec(A), \langle a_1^{e_1} \dots a_c^{e_c} \rangle)$ 
via the morphism $\Spec(B) \to \Spec(A)$.
\end{proof}

\begin{defi}[{\cite[Def.3.1.4]{BPO22}}] \label{defi:logNis}
A commutative square (left)
\[ \xymatrix{
(W, α_W) \ar[r] \ar[d] & (V, α_V) \ar[d] 
&&
W \ar[r] \ar[d] & V \ar[d] 
\\
(U, α_U) \ar[r] & (X, α_X)
&&
U \ar[r] & X
} \]
of log schemes is a  \emph{strict distinguished Nisnevich square} if the underlying square of schemes (right) is a distinguished Nisnevich square, and all morphisms are \emph{strict} in the sense that the log structures $α_W, α_V, α_U$ are the ones induced by $α_X$, Def.\ref{defi:logInd}.
\end{defi}

\begin{exam} \label{exam:distSquaresImage}
Consider the following three sets with $\sX$ a modulus pair.
\begin{enumerate}
 \item distinguished Nisnevich squares over the scheme $\hX$,
 \item distinguished Nisnevich squares over the pair $\sX$,
 \item strict distinguished Nisnevich squares over the log scheme $\log \sX$.
\end{enumerate}
The first two are always in bijection, and Lemma~\ref{lemm:strictReduced} says that the latter two are also in bijection when $\sX \in \PSm^{\nc}_k$.
\end{exam}

\begin{defi}
The \emph{strict Nisnevich topology} on $\SmlSm_k$ is generated by families of the form $\{(U, α_U) \to (X, α_X), (V, α_V) \to (X, α_X)\}$ associated to distinguished Nisnevich squares. The strict Nisnevich topology is denoted $\sNis$.
\end{defi}

\begin{rema} \label{rema:excisionDescent}
Analogous to the Nisnevich topology on categories of schemes, descent (\v{C}ech or hyper) is equivalent to excision for the Nisnevich topology on $\PSm^{\nc}_k$ resp. the strict Nisnevich topology on $\SmlSm_k$. That is, a presheaf (of sets, spaces, complexes of abelian groups, etc.) is a sheaf if and only if it sends distinguished squares to cartesian squares, and the initial object to the terminal object. In both cases this can be deduced from the case of usual schemes by observing that the small Nisnevich site of $\sX \in \PSm^{\nc}_k$ resp. $\log \sX \in \SmlSm_k$ (objects are morphisms towards $\sX$ resp. $\log \sX$ which are reduced, resp. strict, and étale on the total space, resp. underlying scheme) is equivalent to the small Nisnevich site of $\hX$.
\end{rema}

\begin{prop} \label{prop:NisSame}
The canonical functor induces an equivalence
\[ \Shv_{\Nis}(\PSm^{\nc}_k)[\Lambda^{-1}] \stackrel{\sim}{\to} \Shv_{\sNis}(\SmlSm_k) \]
where we identify $\Lambda$ with its image under Yoneda%
\footnote{The Nisnevich topology on $\PSm^{\nc}_k$ is subcanonical, but we don't care about this right now.
} %
and the localisation is happening in the category $\PrL$ of presentable categories with colimit preserving functors.
\end{prop}

\begin{proof}
Colimit preserving functors out of $\Shv_{\Nis}(\PSm^{\nc}_k)[\Lambda^{-1}]$ are the same as colimit preserving functors out of $\Shv_{\Nis}(\PSm^{\nc}_k)$ which invert elements of $\Lambda$. These are the same as functors on $(\PSm^{\nc}_k)^{op}$ which satisfy excision and invert elements of $\Lambda$. These are the same as functors on $(\SmlSm_k)^{op}$ which satisfy excision, Prop.\ref{prop:PSmLambdaSmlSm}, Exam~\ref{exam:distSquaresImage}, and these are the same as colimit preserving functors out of $\Shv_{\sNis}(\SmlSm_k)$.
\end{proof}



\section{Admissible modifications} \label{sec:adb}




\begin{defi} \label{defi:sbu}
In this note, we call a morphism $\sX \to \sY$ a \emph{smooth blowup} 
if $\hX = \Bl_{\hY} Z$ for some closed $Z \subseteq \mX$ for which 
there exists an open covering $\{U_i \to X\}_{i \in I}$ and étale morphisms $U_i \to \AA^d$ with 
\begin{align*}
	\mX|_{U_i} &= \Spec(k[x_1, \dots, x_d] / x_1^{e_1} x_2^{e_2}\dots x_a^{e_a}) |_{U_i}, \textrm{ and } \\
	Z|_{U_i} &= \Spec(k[x_1, \dots, x_d] / \langle x_{b_1}, x_{b_2}, \dots, x_{b_c} \rangle) |_{U_i}
\end{align*}
where 
$\{b_1, \dots, b_c\} \cap \{1, \dots, a\} \neq \varnothing$. 
We will write $\Sigma$ for the class of \emph{compositions} of smooth blowups.
\end{defi}



\begin{defi}
\label{defi:logsbu}
We call a morphism $\log(\sX) \to \log(\sY)$ in $\SmlSm_k$ is a \emph{smooth blowup} if it is the image of a smooth blowup in $\PSm_k^{\nc}$, Def.\ref{defi:sbu}.
We will overload the notation $\Sigma$ by also using it for the class of \emph{compositions} of smooth blowups in $\SmlSm_k$.
\end{defi}


\begin{prop} \label{prop:MSmSmlSmSigma}
There exist canonical equivalences of categories
\begin{align*}
\PSm_k[\Sigma^{-1}][\Lambda^{-1}] &\cong \SmlSm_k[\Sigma^{-1}], \\
\Shv_{\Nis}(\PSm_k)[\Sigma^{-1}][\Lambda^{-1}] &\cong \Shv_{\sNis}(\SmlSm_k)[\Sigma^{-1}].
\end{align*}
\end{prop}

\begin{rema}
The first localisation is in the category of small $\infty$-categories, Def.\ref{defi:catloc}. The second one is in the category $\PrL$ of presentable categories and colimit preserving functors.
\end{rema}

\begin{proof}
The first one follows from Proposition~\ref{prop:PSmLambdaSmlSm}. The second one is analogous to Proposition~\ref{prop:NisSame}.
\end{proof}

\section{Homotopy categories} \label{sec:homCat}

\begin{defi} \label{defi:bcube}
Recall that one writes $\bcube$ for the pair $(\PP^1, \{\infty\}) \in \PSm^{\nc}_k$. More generally, for any pair $\sX = (\hX, \mX)$ we write $\bcube_{\sX}$ for the pair 
\[ (\hX {\times} \PP^1, \ \mX {\times} \PP^1 + \hX {\times} \{\infty\}). \]
We will write $\CI$ for the class of such morphisms.
\end{defi}

\begin{defi} \label{defi:MH}
The category $\MH_k^{\nc} \subseteq \Shv_{\Nis}(\PSm_k^{\nc}, \sS)$ of (normal crossings) \emph{homotopy types with modulus} is the full subcategory of those Nisnevich sheaves of spaces $F$ on $\PSm_k^{\nc}$ which:
\begin{enumerate}
 \item send $\varnothing$ to $\ast$,
 \item send smooth modifications to equivalences, Def.\ref{defi:sbu},
 \item send the morphisms $\bcube_{\sX} \to \sX$ to equivalences, Def.\ref{defi:bcube}.
\end{enumerate}
\end{defi}

\begin{rema} \label{rema:NisExcis}
In other words,
\[ \MH_k^{\nc} = \Shv_{\Nis}(\PSm_k^{\nc}, \sS)[\Sigma^{-1}][\CI^{-1}] \]
where
%
 $\Sigma$ and ${\CI}$ means the image of these classes under Yoneda.
\end{rema}

\begin{defi} \label{defi:lH}
The log homotopy category $\logH_k \subseteq \Shv_{\sNis}(\SmlSm_k, \sS)$ is the full subcategory of those sheaves of spaces $F$ on $\SmlSm_k$ which:
\begin{enumerate}
 \item send $\varnothing$ to $\ast$,
 \item send smooth blowups to equivalences, Def.\ref{defi:logsbu},
 \item send the morphisms $\log \bcube_{\sX} \to \log \sX$ to equivalences.
\end{enumerate}
\end{defi}

\begin{rema}
In the first version of \cite{BPO23} the symbol $\logH_k$ was instead used for $\Shv_{\sNis}(\SmlSm_k, \sS)[\Sigma_{\smod}^{-1}, (\log \CI)^{-1}]$ where $\Sigma_{\smod}$ is the class of compositions of \emph{modifications} rather than smooth blowups. The difference is that instead of $\{b_1, \dots, b_c\} \cap \{1, \dots, a\} \neq \varnothing$ in Definition~\ref{defi:sbu}, one instead asks only that $1 \leq b_1, \dots, b_c \leq a$. That is, instead of $Z$ being contained in an irreducible component of $\sX$, one asks the stronger condition that it is an intersection of irreducible components of $\sX$. It seems likely that newer versions of \cite{BPO23} will use the symbol $\logH_k$ as we are using it.
\end{rema}

\begin{theo} \label{theo:main}
Restriction along $\log: \PSm_k^{\nc} \to \SmlSm_k$ identifies $\logH_k$ with the full subcategory of $\MH_k^{\nc}$ whose objects are those presheaves of spaces $F$ which send morphisms in $\Lambda$ to equivalences. That is, we have an equivalence
\[ \MH_k^{\nc}[\Lambda^{-1}] \stackrel{\sim}{\to} \logH_k \]
in the category $\PrL$ of presentable categories and colimit preserving functors.
\end{theo}

\begin{proof}
Proposition~\ref{prop:NisSame} gives an equivalence $\Shv_{\Nis}(\PSm_k^{\nc}, \sS)[\Lambda^{-1}] \stackrel{\sim}{\to} \Shv_{\sNis}(\SmlSm_k, \sS)$. Then inserting the definitions gives
\[ \xymatrix{
\MH_k^{\nc}[\Lambda^{-1}] 
\ar@{=}[r]^-{\textrm{Def.\ref{defi:MH}}}
& 
\Shv_{\Nis}(\PSm_k^{\nc}, \sS)[\Lambda^{-1}][\Sigma^{-1}][\CI^{-1}] 
\ar@{=}[d]
\\
\logH_k
\ar@{=}[r]^-{\textrm{Def.\ref{defi:lH}}}
 & 
\Shv_{\sNis}(\SmlSm_k, \sS)[\Sigma^{-1}][(\log \CI)^{-1}]
} \] 

\end{proof}

\section{Rational divisors} \label{sec:ratDiv}

Recall that the abelian group of rational numbers can be expressed as the union $\QQ = \cup_{n \in \NN_{\geq 0}} \tfrac{1}{n}\ZZ$, or equivalently, as the colimit $\QQ = \colim_{n \in \NN_{\geq 0}} \ZZ$ where the transition morphism $\ZZ \to \ZZ$ associated to the indexing arrow $n {\to} nm$ is multiplication by $m$.

More generally, consider the $\NN_{\geq 0}$ indexed diagram which sends $m \leq mn$ to $(-)^{(n)}: \PSm_k^{\nc} \to \PSm_k^{\nc}$. The colimit over this diagram is Koizumi--Miyazaki's category 
\begin{equation} \label{equa:mSmcolim}
\mSm_k
 = 
\underset{n \in \NN_{\geq 1}}{\colim}\ \PSm_k^{\nc}
\end{equation}
from \cite{KM24} also featuring in their later work with Saito \cite{KMS25}. Since the $(-)^{(n)}$ are fully faithful, $\mSm_k$ can also be described as the full subcategory $\mSm_k \subseteq \PSh(\PSm_k)$ whose objects are of the form $\sX^{(1/n)} := \hom((-)^{(n)}, \sX)$ for $\sX \in \PSm_k$.%
\footnote{A similar construction was used by Matsumoto in \cite[\S 8]{Mat23}.} %
The Nisnevich topology on $\mSm_k$ has coverings which are coverings in any of the  $\PSm_k^{\nc}$. Similarly, the class $\BI$ of morphisms in $\mSm_k$ is the union of the classes $\Sigma$ in the components $\PSm_k^{\nc}$ of this colimit. The class $\CI \subseteq \mSm_k$ is a little more subtle. 

\begin{defi} \label{defi:CIn}
Define
\begin{align*}
\CI_n &= \biggl \{ 
(\hX {\times} \PP^1, \ \mX {\times} \PP^1 + n \cdot \hX {\times} \{\infty\}) \to (\hX, \mX)
\biggr \}_{(\hX, \mX) \in \PSm_k^{\nc}}, \\
\CI &= \cup_{n \geq 1} \iota_n \CI_n
\end{align*}
where $ι_n: \PSm_k^{\nc} \to \mSm_k$ is the functor induced by the $n$th component of the colimit. Then the category $\mmH_k$ is defined, \cite[Def.1.10]{KMS25}, as the localisation 
\begin{equation} \label{equa:mmH}
\mmH_k = \Shv_{\Nis}(\mSm_k)[\BI^{-1}, \CI^{-1}].
\end{equation}
\end{defi}


\begin{prop} \label{prop:mH}
There is an equivalence of categories 
\begin{equation} \label{equa:mHcolim}
\mmH_k = \underset{n \in \NN_{\geq 1}}{\colim} \MH_k[\CI_n^{-1}] 
\end{equation}
where the transition morphism associated to $n \to nm$ 
is induced by the functor $(-)^{(m)}: \PSm_k^{\nc} \to \PSm_k^{\nc}$. Moreover, if $k$ admits resolution of singularities, the canonical functor associated to the initial object $1 \in \NN$
\begin{equation} \label{equa:MHkmmHk}
\MH_k \to \mmH_k
\end{equation}
is fully faithful.
\end{prop}

\begin{rema} \label{rema:CanUseMSm}
The resolution of singularities assumption is probably unnecessary. It is just for convenience so that we can use the category $\MSm_k$ where morphisms of $\CI_n$ are categorical pullbacks of $(\PP^1, n\{\infty\}) \to (\Spec(k), \varnothing)$. This is because we want to apply \cite[Prop.3.4.(1)]{Hoy17} to describe the localisation at $\CI_{mn}$, so we want the class $\CI_{mn}$ to be closed under categorical pullback.
\end{rema}

\begin{proof}
The equivalence \eqref{equa:mHcolim} is rather straightforward. Let $\sE$ be a presentable infinity category. Then for both $\sC = \mmH_k$ and $\sC = \underset{n \in \NN_{\geq 1}}{\colim} \MH_k[\CI_n^{-1}]$, the category of accessible functors ${\operatorname{Fun}}(\sC, \sE)$ is identified with the full subcategory of objects in  $\lim_{n \in \NN} {\operatorname{Fun}}(\PSm_k^{\nc}, \sE)$ such that the $n$th term inverts morphisms in $\Sigma \cup \CI_n$ and sends Nisnevich squares to cocartesian squares.

For the fully faithfulness claim, it suffices to show that each transition $\MH_k[\CI_n^{-1}] \to \MH_k[\CI_{mn}^{-1}]$ is fully faithful. We do this in steps. 

First, as observed in Remark~\ref{rema:CanUseMSm}, since we are inverting smooth blowups and not just modifications, and the base field has resolution of singularities, we have an equivalence $\PSm_k[\Sigma^{-1}] \cong \MSm_k$, so we can work with this latter category. 

Next using the calculus of fractions description of morphisms in $\PSm_k[\Sigma^{-1}] \cong \MSm_k$ we see that the induced functor $$μ := (-)^{(m)}: \MSm_k^{\nc} \to \MSm_k^{\nc}$$ is fully faithful. It follows that the left Kan extension $μ^*: \PSh(\MSm_k^{\nc}) \to \PSh(\MSm_k^{\nc})$ is fully faithful.

Thirdly, consider the two localisation functors 
\begin{align*}
L_{\Nis}: 
\PSh(\MSm_k^{\nc}) &\to \Shv_{\Nis}(\MSm_k^{\nc}), \\
L_{\CI_{mn}}: 
\PSh(\MSm_k^{\nc}) &\to \PSh(\MSm_k^{\nc})[\CI_{mn}^{-1}].
\end{align*}
Both targets are canonically identified with subcategories of $\PSh(\MSm_k^{\nc})$, so the functors $L_\bigcirc$ can be considered as endofunctors of $\PSh(\MSm_k^{\nc})$. We claim that for any $\Nis$-local, resp. $\CI_n$-local, $F \in \PSh(\MSm_k^{\nc})$ we have $F = μ_*L_\bigcirc μ^* F$ for $\bigcirc = \Nis, \CI_{mn}$ (note the subscript on the first $\CI$ is $n$ and the second one is $mn$).

For $\bigcirc = \Nis$ this follows from the usual description of the sheafification functor as a transfinite iteration of $\check{H}^0F(\sX) = \colim_{R \to \sX} \Map(R, F)$ where the colimit is over covering sieves. For any $\sX$ in the image of $μ$, and Nisnevich covering $\{\sU_i \to \sX\}_{i \in I}$, the morphisms $\sU_i \to \sX$ are also in the image of $μ$, so if $G$ is already a Nisnevich sheaf on the image of $μ$, then $G \to L_{\Nis}G$ is an equivalence on the image of $μ$.

The same argument works for $\bigcirc = \CI_{mn}$ in light of the description $L_{\CI_{mn}}F(\sX) = \colim_{\sY \to \sX} F(\sY)$ where the colimit is over the fullsubcategory of $(\MSm_k^{\nc})_{/\sX}$ whose objects are compositions of morphisms in $\CI_{mn}$, \cite[Prop.3.4.(1)]{Hoy17}. Namely, if $\sX$ is in the image of $μ$ and $\sY \to \sX$ is a composition of morphisms in $\CI_{mn}$, then $\sY \to \sX$ is in the image of $μ$. So if $G$ is already $\CI_{mn}$-local on the image of $μ$, then $G \to L_{\CI_{mn}}G$ is an equivalence on the image of $μ$.

Combining the above two claims we will obtain: if $F$ is both $\Nis$-local and $\CI_n$-local, then $F = μ_*L_{\Nis, \CI_{mn}} μ^* F$ where $L_{\Nis, \CI_{mn}}$ is the localisation functor
\[ L_{\Nis, \CI_{mn}}: \PSh(\MSm_k^{\nc}) \to \Shv_{\Nis}(\MSm_k^{\nc})[\CI_{mn}^{-1}]. \]
Indeed, one sees that $L_{\Nis, \CI_{mn}}$ can be described as the transfinite composition $L_{\Nis, \CI_{mn}} = \colim_{i \in \NN} (L_{\Nis} \circ L_{\CI_{mn}})^{\circ i}$. Consequently, we have shown that the $(n {\to} mn)$-transition functor
\[ 
\underset
{=\MH_k[\CI_n^{-1}]}
{\Shv_{\Nis}(\MSm_k^{\nc})[\CI_{n}^{-1}]}
\to 
\underset
{=\MH_k[\CI_{mn}^{-1}]}
{\Shv_{\Nis}(\MSm_k^{\nc})[\CI_{mn}^{-1}]}
\]
is fully faithful.
\end{proof}

Stabilising, \cite[\S 1.4]{HA}, with respect to the simplicial circle $\ast \underset{\ast \sqcup \ast}{\sqcup} \ast$ gives stable $\infty$-categories $\MSH_k = \Sp(\MH_k)$ and $\mmSH_k = \Sp(\mmH_k)$. Pushing Proposition~\ref{prop:mH} through the stabilisation functor $\Sp(-)$ gives a stable version.

\begin{coro} \label{coro:MSHvsmSH}
There is an equivalence of categories 
\begin{equation} 
\mmSH_k = \underset{n \in \NN_{\geq 1}}{\colim} \MSH_k[\CI_n^{-1}] 
\end{equation}
where the transition morphism associated to $n \to nm$ 
is induced by the functor $(-)^{(m)}: \PSm_k^{\nc} \to \PSm_k^{\nc}$. Moreover, if $k$ admits resolution of singularities, the canonical functor associated to the initial object $1 \in \NN$
\begin{equation} \label{equa:MHkmmHk}
\MSH_k \to \mmSH_k
\end{equation}
is fully faithful.
\end{coro}

\section{Transfers} \label{sec:transfers}

In this section we show in Example~\ref{exam:gmgm} that the transfers version of Proposition~\ref{prop:PSmLambdaSmlSm} fails.

\begin{defi}[{\cite{SVRelCyc}, \cite{VoeTri}, \cite[0H4B]{stacks-project}}] \label{defi:corr}
For $X, Y$ essentially%
\footnote{By \emph{essentially smooth} we mean a filtered intersection $\cap U_λ$ of open subschemes $U_λ \subseteq X$ of a smooth scheme $X$ whose transition morphisms $U_λ \to U_μ$ are affine.} %
smooth $k$-schemes, the abelian group of \emph{correspondences} from $X$ to $Y$ is the free abelian group
\[ \Cor(X, Y) = \ZZ \{ Z \subseteq X \times Y\ |\ Z \stackrel{\textrm{fin.surj.}}{\to} X \} \]
generated by closed integral subschemes $Z \subseteq X \times Y$ such that the projection $Z \to X \times Y \to X$ is finite and surjective over an irreducible component of $X$. The class of such a $Z$ is denoted $[Z]$. Correspondences of the form $[Z]$ are called \emph{prime}.
\end{defi}

\begin{exam}\ 
\begin{enumerate}
 \item 	(Graph) Given a morphism $X \stackrel{f}{\to} Y$ of essentially smooth schemes with $X$ connected, one defines the graph $[f]\in \Cor(X, Y)$ as the class $[(\id{\times}f)(X)]$ of the image of $X$ under the canonical morphism $\id {\times} f: X \to X \times Y$. If $X$ is not connected one defines $[f]$ as the sum of the graphs of the connected components of $X$.
 
 \item (Flat pullback) If 
$[Z] \in \Cor(X, Y)$ is prime correspondence 
and
$p: ξ \to X$ is a morphism from the spectrum $ξ$ of a field such that $Z \to X$ is flat at $p(ξ)$ one defines
\[ p^*[Z] := \sum_{i \in I} m_i[z_i] \]
where $z_i$ are points of $Z' = ξ \times_X Z$ and $m_i = \length \OO_{Z', z_i}$. %
For correspondences $\sum n_i [Z_i]$ such that each $Z_i \to X$ is flat at $p(ξ)$ we set $p^*\sum n_i [Z_i] = \sum n_i p^*[Z_i]$.
%
%
\end{enumerate}
\end{exam}

In \cite{SVRelCyc}, Suslin and Voevodsky construct a family of morphisms 
\[ \Cor(W, X) \times \Cor(X, Y) \to \Cor(W, Y); \qquad (\alpha, \beta) \mapsto \beta \circ \alpha \]
associated to triples of essentially smooth $k$-schemes $W, X, Y$ which is uniquely determined by the following properties, cf.\cite[Chap.2]{Kel13},\footnote{Actually, Suslin-Voevodsky construct compositions satisfying (and uniquely determined by) (1) and (2) for all reduced Noetherian schemes, however in general, $\Cor(X, Y)$ is a proper subgroup of the free abelian group from Definition~\ref{defi:corr}. The theory is extended to non-reduced Noetherian schemes by setting $\Cor(X, Y) := \Cor(X_{\red}, Y_{\red})$.}
\begin{enumerate}
 \item $(\gamma \circ \beta) \circ \alpha = \gamma \circ (\beta \circ \alpha)$.
 \item If $α = n_i [Z_i] \in \Cor(X, Y)$ is a correspondence and $p: ξ \to X$ is a morphism from the spectrum of a field such that each $Z_i \to X$ is flat at $p(ξ)$, then 
 \[ \alpha \circ [p] = p^*α. \]
 \end{enumerate}

\begin{rema}
Note that it follows directly from (1) and (2) that for any $Y$ and dominant morphism $p: X' \to X$ the corresponding morphism $p^*: \Cor(X, Y) \to \Cor(X', Y)$ is injective.
\end{rema}

\begin{defi}[{\cite{KMSYI}, \cite{KMSYII}}] \label{defi:MCor}
Given $\sX, \sY \in \PSm^{\nc}_k$, one can define $\MCor(\sX, \sY)$ as the subgroup of those correspondences $\sum n_i [Z_i] \in \Cor(\iX, \iY)$ such that each $\hZ_i \to \hX$ is proper and 
\[ \mX|_{\hZ_i^N} \supseteq \mY|_{\hZ_i^N} \]
for each $i$ where $\hZ_i^N \to \hZ_i$ is the normalisation of the closure $\hZ_i$ of $Z_i$ in $\hX \times \hY$. Composition is inherited from $\Cor_k$. We write $\MCor_k^{\nc}$ for the resulting category.
\end{defi}

\begin{rema}
Observe that the functor $\Lambda: \sX \mapsto \sX^{(n)}$ from Proposition~\ref{prop:PSmLambdaSmlSm} is functorial in $\MCor_k^{\nc}$ and the class of morphisms 
\[ \Lambda_{\Cor} = \{\sX^{(n)} \to \sX : n \geq 1, \sX \in \MCor_k^{\nc}\} \]
 continues to satisfy a right calculus of fractions for the same reasons the $\PSm^{\nc}_k$ version does.
\end{rema}

\begin{defi}[{\cite[Def.2.1.1]{BPO22}}] \label{defi:finCor}
For $\log \sX, \log \sY$ in $\SmlSm_k$ the subgroup 
\[ \lCor(\log \sX, \log \sY) \subseteq \Cor(\hX, \hY) \]
is defined as the set of those $\sum n_i [\hZ_i] \in \Cor(\hX, \hY)$ such that for each normalisation $\hZ_i^N$ we have an inclusion 
\begin{equation} \label{equa:XZYZ}
p^*_{\log}M_{\hX} \supseteq im(q^{-1}M_{\hY})
\end{equation}
of subsheaves of $\OO_{\hZ_i^N}$ where $p: \hZ_i^N \to \hX$ and $q: \hZ_i^N \to \hY$ are the induced morphisms.
\end{defi}

\begin{rema}
We saw in Remark~\ref{rema:indSub} that $p^*_{\log}M_{\hX}$ is a subsheaf of $\OO_{\hZ_i^N}$.
\end{rema}

\begin{rema} \label{rema:logXYinjCor}
We already observed that $\Cor(\hX, \hY) \to \Cor(\iX, \hY)$ is injective. The flat pullback map $\Cor(\iX, \hY) \to \Cor(\iX, \iY)$ is not injective in general, but since 
$\OO_{\iY} = \OO_{\hY}[M_{\hY}^{-1}]$, existence of a factorisation Eq.\eqref{equa:XZYZ} implies that the generic points of the $\hZ_i$ are in $\iX \times \iY$. So
\[ \lCor(\log \sX, \log \sY) \subseteq \Cor(\iX, \iY). \]
Definition~\ref{defi:finCor} can therefore be altered to read: a correspondence $\sum n_i [Z_i] \in \Cor(\iX, \iY)$ is in $ \lCor(\log \sX, \log \sY)$ if and only if for each $i$, 
\begin{enumerate}
 \item the projection $\hZ_i \to \hX$ is finite, where $\hZ_i$ is the closure of $Z_i$ in $\hX \times \hY$.
 \item we have the inclusion Eq.\eqref{equa:XZYZ} in Definition~\ref{defi:finCor}.
\end{enumerate}
\end{rema}

\begin{defi}[{\cite[Rem.4.6.2, Prop.4.6.3, Lem.4.4.3]{BPO22}}]
The subgroup of \emph{dividing log correspondences} is the colimit
\[ \lCor^{\div}(\log \sX, \log \sY) 
=
 \colim_{\Sigma_{\lmod}} \lCor(\log \sX', \log \sY) \]
over compositions of log modifications with a smooth centres $\log \sX' \to \log \sX$. 
\end{defi}

\begin{rema} \label{rema:suboo}
Since all $\lCor(\log \sX', \log \sY) \to \Cor((X')^\circ, \iY)$ are injective, Rem.\ref{rema:logXYinjCor}, and $(\iX)' \cong \iX$ for any $\sX' \to \sX$ in $\Sigma_{\lmod}$, we immediately get an inclusion
\[ \lCor^{\div}(\log \sX, \log \sY) \subseteq \Cor(\iX, \iY). \] 
\end{rema}

\begin{exam} \label{exam:gmgm}
We give an example to show that in general
\begin{equation} \label{equa:MCorLog}
\colim_{n \geq 1} \MCor_k^{\nc}(\sX^{(n)}, \sY) \neq \lCor(\log \sX, \log \sY).
\end{equation}
Our example $\sX = (\AA^1, \{0\}) = \sY$ and the cycle $x^2 = y^3$ in $\Cor(\AA^1 \setminus \{0\}, \AA^1 \setminus \{0\})$. This is in the left hand side, but not in the right hand side of Eq.\eqref{equa:MCorLog}. We now explain why.

Suppose that $\hX$ and $\hY$ are smooth connected curves so that we can ignore modifications and blowups. Let $\hW \to \hX$ be a finite morphism from a smooth connected curve, and $\hW \to \hY$ a morphism such that $\hW \to \hX \times \hY$ is generically an inclusion, so $\hW$ is the normalisation of its image. If $\hW \to \hY$ is constant then we get an element in the groups of Eq.\eqref{equa:MCorLog} if and only if it factors through $\iY$, so let's ignore this case, and assume that $\hW \to \hY$ is not constant.
The condition for the corresponding cycle in $\Cor(\iX, \iY)$ to lie in the left or right hand side of Eq.\eqref{equa:MCorLog} is local on $\hW$ so choose a closed point $w \in \hW$, and let $x, y$ be the images of $w$ in $\hX, \hY$. We obtain corresponding extensions of dvrs $\OO_{\hX, x} \subseteq \OO_{\hW, w}$, $\OO_{\hY, y} \subseteq \OO_{\hW, w}$. Let $e_x, e_y$ be their ramification degrees. Finally, let $n_x, n_y$ be the coefficients of $\mX, \mY$ at $x, y$. Then:
\begin{enumerate}
 \item $\hW$ defines a prime correspondance in $\MCor_k^{\nc}(\sX, \sY)$ if and only if $$n_xe_x \geq n_y e_y$$ for all $w$.

 \item $\hW$ defines a prime correspondance in $\lim_{n \geq 1} \MCor_k^{\nc}(\sX^{(n)}, \sY)$ if and only if $$n_x = 0 \Rightarrow n_y = 0$$ for all $w$.

 \item $\hW$ defines a prime correspondance in $\lCor_k(\log \sX, \log \sY)$ if and only if $$n_xe_x | n_y e_y$$ for all $w$, cf.Remark~\ref{rema:indSub}, Remark~\ref{rema:logXYinjCor}.
\end{enumerate}
Notice that (3) implies (2). For a concrete example of a situation where (2) holds but (3) does not, consider $\sX = (\AA^1, \{0\})$ and $\sY = (\AA^1, \{0\})$, let $\hW = \AA^1$, and take the morphisms coming from 
\begin{align*}
\AA^1 &\to \AA^1 \times \AA^1 \\
t &\mapsto (t^a, t^b)
\end{align*}
with $a, b > 0$ coprime.
In this case, we get a prime divisor of $\Cor(\AA^1 \setminus \{0\}, \AA^1 \setminus \{0\})$ which is in $\colim_{n \geq 1} \MCor_k^{\nc}(\sX^{(n)}, \sY)$ but not in $\lCor(\log \sX, \log \sY)$ whenever $a > 1$.
\end{exam}

\bibliographystyle{amsalpha}
\bibliography{bib}

\end{document}